\documentclass[11pt]{article}

\usepackage[T1]{fontenc}
\usepackage{lmodern}
\usepackage{microtype}
\usepackage{amsmath,amssymb,mathtools,amsthm}
\usepackage{geometry}
\usepackage{enumitem}
\usepackage{aliascnt}
\usepackage{xcolor}
\usepackage[colorlinks=true,linkcolor=blue!55!black,citecolor=blue!55!black,urlcolor=blue!55!black]{hyperref}
\usepackage[nameinlink,capitalise]{cleveref}

\allowdisplaybreaks

\newtheorem{theorem}{Theorem}
\newaliascnt{corollary}{theorem}
\newtheorem{corollary}[corollary]{Corollary}
\aliascntresetthe{corollary}
\newaliascnt{lemma}{theorem}
\newtheorem{lemma}[lemma]{Lemma}
\aliascntresetthe{lemma}
\newaliascnt{proposition}{theorem}
\newtheorem{proposition}[proposition]{Proposition}
\aliascntresetthe{proposition}
\theoremstyle{definition}
\newaliascnt{definition}{theorem}

\aliascntresetthe{definition}
\theoremstyle{remark}
\newaliascnt{remark}{theorem}

\aliascntresetthe{remark}

\crefname{theorem}{Theorem}{Theorems}
\Crefname{theorem}{Theorem}{Theorems}
\crefname{corollary}{Corollary}{Corollaries}
\Crefname{corollary}{Corollary}{Corollaries}
\crefname{lemma}{Lemma}{Lemmas}
\Crefname{lemma}{Lemma}{Lemmas}
\crefname{proposition}{Proposition}{Propositions}
\Crefname{proposition}{Proposition}{Propositions}
\crefname{definition}{Definition}{Definitions}
\Crefname{definition}{Definition}{Definitions}
\crefname{remark}{Remark}{Remarks}
\Crefname{remark}{Remark}{Remarks}

\DeclareMathOperator{\rk}{r}
\DeclareMathOperator{\cl}{cl}
\DeclareMathOperator{\tr}{tr}
\newcommand{\T}{\mathcal{T}}
\newcommand{\Q}{\mathcal{Q}}

\title{Piercing spanned lines and hyperplanes}
\author{%
    Ben Lund\thanks{Institute of Mathematics and Interdisciplinary Sciences,
Xidian University, Xi'an, China.
Email: \href{mailto:lund.ben@gmail.com}{\texttt{lund.ben@gmail.com}}.}}
\date{}
\hypersetup{pdfauthor={Ben Lund}}

\begin{document}
\maketitle

\begin{abstract}
Motzkin, Gr\"unbaum, and Erd\H{o}s and Purdy asked how many
noncollinear blue points can coexist with $t$ red points in the real
plane if every blue-spanned line contains a red point. We prove that
there are at most $2.35t-3$ blue points, even over the complex numbers.
For the complex higher-dimensional problem of Campbell, Kroeker, and Lund,
we prove $b_k(t)\le1.40k^2t$. Colorings of cyclic Dowling
geometries give a matching lower bound in order of magnitude for
all $k\ge2$ and $t\ge3$. The planar proof uses frame normalization
and a second-moment estimate; the higher-dimensional proof combines
the degree bound of Kroeker and Lund with principal truncation.
\end{abstract}

\section{Introduction}

Let $B$ and $R$ be disjoint finite sets of blue and red points in the
real plane, with $B$ noncollinear. How large can $|B|$ be in terms of
$|R|$ if every line spanned by two blue points contains a red point?
This problem was considered by Motzkin~\cite{Motzkin1951,Motzkin1967},
Gr\"unbaum~\cite{Grunbaum1975}, and
Erd\H{o}s and Purdy~\cite{ErdosPurdy1978}. We study it over the
complex numbers and obtain bounds both for the planar problem and the higher-dimensional
generalization posed by Campbell--Kroeker--Lund~\cite{CKL}.

The planar problem is related to the Motzkin--Rabin theorem, which
states that every finite noncollinear set of red and blue points in
the real plane spans a monochromatic line; see
Motzkin~\cite{Motzkin1967} and Chakerian~\cite{Chakerian1970}.
Requiring every monochromatic line to be red leads to the
{biased colorings} studied by Motzkin and Gr\"unbaum.
Equivalently, the red points must pierce all lines spanned by the
blue points.

Write $n=|B|$ and $t=|R|$.
Motzkin~\cite{Motzkin1951} proved $n\le(t-1)^2$ over $\mathbb R$.
The weak Dirac theorem, proved independently by Beck~\cite{Beck1983}
and Szemer\'edi and Trotter~\cite{SzemerediTrotter1983}, gives a linear
bound: some blue point is incident to a positive proportion of $n$
blue-spanned lines, and these lines require distinct red points.
Pinchasi~\cite{Pinchasi2013} proved $n\le3t+1$ by an elementary
argument that avoids point--line incidence bounds.
Huicochea, Lea{\~n}os, and Rivera~\cite{HLR2022} stated a further
refinement of this bound.

Some additional restrictions permit stronger bounds. If the red
points are collinear, their supporting line is disjoint from $B$
and can be taken as the line at infinity. The problem then becomes
one of counting directions determined by $B$, and Ungar's
theorem~\cite{Ungar1982} gives $t\ge2\lfloor n/2\rfloor$.
Without this restriction, Gr\"unbaum~\cite{Grunbaum1975,Grunbaum1999}
constructed real configurations with $n=t+4$ for infinitely many
$t$, as well as a configuration with $n=16$ and $t=10$.

Over $\mathbb C$, the degree consequence of
Langer's theorem~\cite[Proposition~11.3.1]{Langer2003}
obtained by Han~\cite{Han2017WeakDirac} implies $n\le3(t-1)$:
some blue point lies on at least $(n+3)/3$ blue-spanned lines,
which again require distinct red points. We improve this bound to
$n\le(8-4\sqrt2)t-3$.
In the other direction, the construction in \cref{sec:construction}
gives complex planar configurations with $n=2t-3$ for every $t\ge3$.

The higher-dimensional problem introduced by Campbell, Kroeker,
and Lund~\cite{CKL} arose from their study of the average number of
points in a spanned hyperplane. They showed that, in fixed dimension,
a configuration with sufficiently large average hyperplane size
has every spanned hyperplane containing one of a bounded collection
of spanned lines. Moreover, the point set is the union of a \emph{degenerate}
set contained in a union of flats whose total dimension is less than
the ambient dimension with an exceptional set of bounded size. Their colored
extremal function gives a sharp bound on the size of this exceptional set.
The connection is illustrated by replacing each red point with
many points on a line in a new independent direction: the piercing
condition forces every spanned hyperplane of the resulting
configuration to contain one of these lines.

Simply replacing lines by hyperplanes in the planar problem does
not give a finite extremal function, even when the blue points span
the ambient space. In three dimensions, place arbitrarily many blue
points and one red point on each of two skew lines. Every plane
spanned by the configuration contains one of the two lines, and
hence a red point. Thus two red points suffice for arbitrarily many
blue points. The additional condition (C3) below excludes this
example and supplies the complements needed for contraction
arguments.

For their higher-dimensional formulation, Campbell, Kroeker, and
Lund~\cite[Theorem~1.3]{CKL} proved an upper bound of $O_k(t)$;
their argument also applies over $\mathbb C$. We give an explicit
$O(k^2t)$ upper bound over $\mathbb C$ and constructions showing that
this is the correct order of magnitude, uniformly for all $k\ge2$
and $t\ge3$.

We use matroid terminology to describe finite complex vector
configurations. A complex-representable matroid $M$ records the
linear dependencies among vectors $(v_e)_{e\in E}$, where $E$ is its
\emph{ground set}. For $X\subseteq E$, its \emph{rank}
$\rk_M(X)$ is the dimension of $\operatorname{span}\{v_e:e\in X\}$;
the rank of $M$ is $\rk(M)=\rk_M(E)$. A \emph{basis} of $X$ is a
subset whose vectors form a basis of this span, and $X$ \emph{spans}
$M$ if $\rk_M(X)=\rk(M)$. A zero vector represents a \emph{loop},
and nonzero scalar multiples represent \emph{parallel} elements.
The matroid is \emph{simple} if it has neither loops nor parallel
elements, so a simple rank-$r$ complex-representable matroid can be
viewed as a finite spanning set of distinct points in
$\mathbb{CP}^{r-1}$.

The \emph{closure} $\cl_M(X)$ consists of all $e\in E$ whose vectors
lie in the span of the vectors indexed by $X$. A \emph{flat} is a
set equal to its closure. A \emph{point}, \emph{line}, and
\emph{hyperplane} mean a flat of rank one, two, and $\rk(M)-1$,
respectively. Thus a flat contains every element of the configuration
in its span, rather than every point of the ambient projective space.
The \emph{restriction} $M|X$ retains only the vectors indexed by $X$.
Sets $X,Y$ are \emph{skew} if their vector spans intersect only in
zero, equivalently if
$\rk_M(X\cup Y)=\rk_M(X)+\rk_M(Y)$.

A coloring of a simple
rank-$r$ matroid $M$ is a partition $(R,B)$ of its ground set into
red and blue sets. Following Campbell, Kroeker, and Lund~\cite{CKL},
we call the coloring \emph{admissible} if
\begin{enumerate}[label=(C\arabic*),leftmargin=*,itemsep=2pt,topsep=5pt]
    \item $\rk_M(B)\ge2$;
    \item every hyperplane of $M$ contains a red point;
    \item for every line $L$ of $M|B$, there is a rank-$(r-2)$
    flat $G$ of $M$ skew to $L$ such that $G\cap R=\varnothing$.
\end{enumerate}
Conditions (C1) and (C3) imply that $B$ spans $M$: a blue line and
its red-free skew complement both lie in $B$ and together have rank
$r$. A \emph{blue flat} means a flat of $M|B$.

Condition (C3) has both a geometric and an inductive role.
In rank three, it says that every blue-spanned line has a blue
point outside it, so it is exactly the noncollinearity condition
from the planar problem. In the two-skew-line example, it fails
for the blue points on either original line: the only line of the
configuration skew to that line is the other original line, which
contains a red point.

More generally, (C3) supplies a red-free complement to every blue
flat $F$ of rank $q\ge2$: there is a flat $P$ of rank $r-q$ skew
to $F$ and containing no red point; see \cref{lem:complement}.
Contracting $P$ means taking the quotient by its vector span.
This reduces the ambient rank to $q$ while preserving all ranks
within $F$. No red point becomes a loop, and every hyperplane of
the quotient contains a red point, since it lifts to a hyperplane
of $M$ containing $P$. Thus (C3) lets us study a blue flat in its
own rank while retaining the hyperplane-piercing condition.
The resulting coloring need not satisfy (C3); principal
truncations repair such failures in the inductive argument.

For $\mathbb K\in\{\mathbb R,\mathbb C\}$ and integers $k,t\ge1$,
let $b_k^{\mathbb K}(t)$ be the maximum number of blue points in an
admissible $\mathbb K$-representable matroid of rank $k+1$ with $t$
red points, with value zero when no such matroid exists. Then
\[
    b_k^{\mathbb R}(t)\le b_k^{\mathbb C}(t).
\]
We henceforth write $b_k(t)=b_k^{\mathbb C}(t)$ and work over
$\mathbb C$; every upper bound also applies over $\mathbb R$.
Since a finite projective configuration can be placed in an affine
chart, $b_2(t)$ is the complex planar extremal function.
By \cref{lem:three-red}, both extremal functions vanish for $t=1,2$.

\begin{theorem}\label{thm:main}
For all integers $k\ge2$ and $t\ge3$,
\begin{align*}
    b_k(t)
    &\le\frac{k+1}{3}\bigl((7-2\sqrt2)k-6\bigr)t
        -\frac{(k+1)(3k-4)}2\\
    &\le\frac{7-2\sqrt2}{3}k^2t.
\end{align*}
\end{theorem}

For $k=2$, the bound is
\[
    b_2(t)\le(8-4\sqrt2)t-3\qquad(t\ge3).
\]
In \cref{sec:construction}, we color cyclic Dowling geometries to obtain
\[
    b_k(t)\ge
    (t-2)\left(\binom{k+1}{2}-1\right)+k-1
    \qquad(k\ge2,\ t\ge3).
\]
Consequently, $b_k(t)=\Theta(k^2t)$ with absolute implied constants.

\subsection*{Proof sketch.}
We first prove the planar bound by comparing local and global
second moments of normalized point vectors. For higher dimensions,
we induct on rank. The degree bound of Kroeker and
Lund~\cite[Theorem~1.3]{KLFrame} lets us contract a blue point while
retaining many blue points. Principal truncations restore
admissibility, and a rank identity controls the points merged
during these truncations. Applying induction to the corresponding
blue flats yields a recurrence with quadratic growth in the rank.

\subsection*{Acknowledgments}

The author thanks Tao Zhang and Tuan Tran for productive discussions
at the outset of this project. OpenAI ChatGPT-6 Astra was used to
help complete the proofs and assist with writing. The author takes
full responsibility for the correctness of the results and the
content of the paper.

\section{The planar bound}\label{sec:planar}

The proof normalizes the blue point vectors so that their frame
operator is nearly a scalar matrix. On each blue-spanned line, the
smallest nonzero eigenvalue bounds the contribution of a red point
on that line. Summing these bounds counts each blue--red pair at
most once, while the total eigenvalue error is controlled by
counting blue pairs. Cauchy--Schwarz then reduces the argument to
a one-variable inequality.

The combination of frame normalization and second-moment estimates appears in earlier work of Dvir, Saraf, and Wigderson on three-query locally correctable codes \cite{DSW3LCC}, as well as in work of Kroeker and Lund \cite{KLFrame} and de Borbon and Panov \cite{deBorbonPanov2026Hirzebruch}.
Here we adapt this method to the colored setting by using local eigenvalue bounds to control the contribution of a red point on each blue-spanned line.

For a matrix $U$, write $U^*$ for its conjugate transpose; for
Hermitian matrices, $S\preceq T$ means that $T-S$ is positive
semidefinite. The following frame-normalization lemma is a
consequence of the complex version of Barthe's radial isotropy
theorem~\cite{Barthe}; see Franks~\cite[Section~6.3, Theorems~75 and~77]{Franks}
for the complex formulation. Kroeker and Lund~\cite[Appendix~A]{KLFrame}
give a self-contained proof of this form.

\begin{lemma}[{Kroeker--Lund~\cite[Lemma~2.1]{KLFrame}}]\label{lem:frame}
Suppose $v_1,\ldots,v_n$ are nonzero vectors spanning $\mathbb C^D$,
and let $T\ge n/D$. Assume that
\[
    |\{i:v_i\in W\}|\le T\dim W
    \qquad\text{for every linear subspace }W\subseteq\mathbb C^D.
\]
The vectors need not be distinct or pairwise nonparallel; the count
is by index. Given $\eta>0$, there is an invertible linear map $A$
for which
\[
    u_i=\frac{Av_i}{\|Av_i\|},\qquad
    \sum_{i=1}^n u_iu_i^*\preceq(T+\eta)I_D.
\]
\end{lemma}

\begin{theorem}\label{thm:planar}
Let \(B,R\subseteq\mathbb{CP}^2\) be disjoint finite sets, with \(B\)
noncollinear. If every line spanned by \(B\) meets \(R\), then
\[
    |B|\le(8-4\sqrt2)|R|-3.
\]
\end{theorem}

\begin{proof}
Write \(n=|B|\) and \(m=|R|\). If a line contains \(h>2n/3\)
blue points, choose a blue point outside it. The \(h\) joining lines
require distinct red points, so \(m\ge h\). This implies the asserted
bound for \(n\ge6\), since
\[
    \frac{2n}{3}\ge\frac{2+\sqrt2}{8}(n+3).
\]
For \(n=4,5\), respectively \(h\ge3,4\), which also suffice;
for \(n=3\), this case cannot occur.

We may therefore assume that every line contains at most \(2n/3\)
blue points. Choose nonzero homogeneous representatives of the blue
points in $\mathbb C^3$. They span $\mathbb C^3$, and a subspace of
dimension one or two contains at most $1\le n/3$ or $2n/3$ of them,
respectively. Thus \cref{lem:frame} applies with $T=n/3$.
Fix $\eta>0$, apply its transformation to all blue and red
representatives, and normalize them to unit vectors $u_b$ and $v_r$.
The blue frame operator satisfies
\[
    S:=\sum_{b\in B}u_bu_b^*
    \preceq\left(\frac n3+\eta\right)I_3,
    \qquad \tr S=n.
\]
Let \(\mathcal L\) be the set of blue-spanned lines. For \(L\in\mathcal L\),
put
\[
    b_L=|B\cap L|,\qquad
    S_L=\sum_{b\in B\cap L}u_bu_b^*,\qquad
    \varepsilon_L=\tr(S_L^2)-\frac{b_L^2}{2}.
\]
The operator \(S_L\) has rank two and trace \(b_L\), so
\(\varepsilon_L\ge0\), and its two nonzero eigenvalues are
\(b_L/2\pm\sqrt{\varepsilon_L/2}\).
Choose one red point \(r_L\in R\cap L\). Since its representative
\(v_{r_L}\) is a unit vector in the range of \(S_L\), its Rayleigh
quotient is at least the smaller nonzero eigenvalue:
\[
    v_{r_L}^*S_Lv_{r_L}
    \ge\frac{b_L}{2}-\sqrt{\frac{\varepsilon_L}{2}}.
\]
Expanding the local quadratic forms gives
\[
    \sum_{L\in\mathcal L}v_{r_L}^*S_Lv_{r_L}
    =\sum_{L\in\mathcal L}\sum_{b\in B\cap L}|u_b^*v_{r_L}|^2
    \le\sum_{r\in R}\sum_{b\in B}|u_b^*v_r|^2
    =\sum_{r\in R}v_r^*Sv_r.
\]
Indeed, each pair $(b,r)\in B\times R$ occurs in the local sum
at most once: if $b\in L$ and $r=r_L$, then $L$ is the unique
line through $b$ and $r$. All summands are nonnegative, so including
the remaining pairs can only increase the sum. Combining this with
the eigenvalue bound above and the global frame bound yields
\begin{equation}\label{eq:cross-energy}
    \sum_{L\in\mathcal L}
        \left(\frac{b_L}{2}-\sqrt{\frac{\varepsilon_L}{2}}\right)
    \le\sum_{r\in R}v_r^*Sv_r
    \le m\left(\frac n3+\eta\right).
\end{equation}

Write \(I=\sum_Lb_L\), \(V=|\mathcal L|\), and
\(a=n(n+3)/3\). Each ordered pair of distinct blue points lies on
exactly one line of \(\mathcal L\), so
\[
    \sum_Lb_L^2=n(n-1)+I,\qquad
    \sum_L\tr(S_L^2)=\tr(S^2)+I-n.
\]
Since \(\tr(S^2)\le n(n/3+\eta)\), it follows that
\begin{equation}\label{eq:frame-excess}
    0\le\sum_L\varepsilon_L
    =\tr(S^2)+\frac I2-\frac{n^2+n}{2}
    \le\frac{I-a}{2}+n\eta.
\end{equation}
Cauchy--Schwarz and \cref{eq:cross-energy,eq:frame-excess} now give
\[
    I-\sqrt{V(I-a+2n\eta)}
    \le2m\left(\frac n3+\eta\right).
\]
The incidence counts are independent of \(\eta\). Letting
\(\eta\downarrow0\), \cref{eq:frame-excess} gives \(I\ge a\), and
the last inequality gives
\[
    I-\sqrt{V(I-a)}\le\frac{2nm}{3}.
\]
As \(V\le I/2\), we obtain
\begin{equation}\label{eq:planar-scalar}
    I-\sqrt{\frac{I(I-a)}2}\le\frac{2nm}{3}.
\end{equation}
For \(x\ge a\),
\[
    x-\sqrt{\frac{x(x-a)}2}\ge\frac{2+\sqrt2}{4}a.
\]
Indeed, \(x-(2+\sqrt2)a/4\ge0\), and
\[
    \left(x-\frac{2+\sqrt2}{4}a\right)^2
        -\frac{x(x-a)}2
    =\frac12\left(x-\frac{1+\sqrt2}{2}a\right)^2\ge0.
\]
Applying this to \cref{eq:planar-scalar} yields
\(m\ge(2+\sqrt2)(n+3)/8\), as required.
\end{proof}

\section{Higher-dimensional reductions}\label{sec:higher-reductions}

We collect the matroid reductions used in the higher-dimensional
induction. The \emph{contraction} $M/X$ is represented by the images
of the vectors indexed by $E\setminus X$ in the quotient by the span
of the vectors indexed by $X$. \emph{Simplification} removes loops
and retains one representative from each parallel class.
We simplify colored matroids by giving red priority: a parallel class
is red if it contains a red element, and blue otherwise.
In a loopless matroid that is not simple, a \emph{blue point} means a
rank-one flat containing no red element. Conditions (C1)--(C3) are then
interpreted in the corresponding simplification.

A \emph{circuit} is a minimally linearly dependent set, and a
\emph{cocircuit} is the complement in $E$ of a hyperplane.
The \emph{connected components} are the equivalence classes generated
by pairs of elements lying in a common circuit. The spans of distinct
components form a direct sum.

\begin{lemma}[{\cite[Proposition~2.3]{CKL}}]\label{lem:three-red}
Every admissible coloring of a simple matroid has at least three
red points.
\end{lemma}

\begin{proof}
Condition (C2) excludes an empty red set. If there were a single red
point, extend it to a basis; the closure of the other basis elements
would be a hyperplane containing no red point.
Suppose instead that the red points are \(r_1,r_2\).
By (C2), no cocircuit contains both of them. Two distinct elements
lie in a common cocircuit if and only if they belong to the same
connected component; see~\cite[Chapter~4]{Oxley2011}.
Thus \(r_1,r_2\) lie in distinct components \(C_1,C_2\).
Since \(B\) spans \(M\), the blue points in \(C_1\) span that
component. Simplicity and the fact that the blue points span \(C_1\) imply \(\rk_M(C_1)\ge2\), so \(C_1\)
contains a blue line \(L\). If a rank-\((\rk(M)-2)\) flat \(G\)
is skew to \(L\), then \(G\cup L\) spans \(M\). As \(L\subseteq C_1\),
the flat \(G\) must contain every other component, including \(r_2\).
This contradicts (C3).
\end{proof}

\begin{lemma}\label{lem:complement}
Let \((R,B)\) be an admissible coloring of \(M\).
If \(F\) is a rank-\(q\) flat of \(M|B\) with \(q\ge2\), then there is a
red-free flat \(P\) of rank \(\rk(M)-q\) skew to \(F\).
Moreover, \((M/P)|F=M|F\).
\end{lemma}

\begin{proof}
Choose a line \(L\) of \(M|F\), and use (C3) to obtain a red-free
rank-\((\rk(M)-2)\) flat \(G\) skew to \(L\).
Then \(F\cup G\) spans \(M\). Choose \(I\subseteq G\) that is a basis
of \(G\) modulo \(F\), and put \(P=\cl_M(I)\). The flat \(P\) has rank
\(\rk(M)-q\), is skew to \(F\), and, since it is contained in \(G\), it is
red-free. Since \(P\) is skew to \(F\), contracting \(P\) preserves the
rank of every subset of \(F\).
\end{proof}

Contraction of a red-free flat preserves (C2), because hyperplanes in
the contraction lift to hyperplanes containing that flat and no red
element becomes a loop. Also, a spanning restriction retaining every
red element preserves (C2): each hyperplane of the restriction closes
to a hyperplane of the original matroid.

In particular, every blue line \(L\) in an admissible matroid has at most
\(t\) points. To see this, take its complement \(P\) from
\cref{lem:complement}. For each \(x\in L\), the set
\(\cl_M(P\cup\{x\})\) is a hyperplane and hence contains a red point.
These hyperplanes intersect pairwise exactly in the red-free flat \(P\),
so they require distinct red points.
This is also~\cite[Lemma~2.4(1)]{CKL}.

The degree bound of Kroeker and Lund~\cite[Theorem~1.3]{KLFrame}
states that a set of $n$ points in $\mathbb C^d$ of affine dimension
$d\ge2$ has a point incident to at least
\[
    n+1-\frac{2n}{d+1}
\]
spanned lines. Its contraction consequence is the following.

\begin{lemma}\label{lem:dirac}
Let \(N\) be a simple complex-representable matroid of rank
\(q\ge3\), with a coloring satisfying (C2), at most \(t\) red points,
and a spanning blue set of size \(n\).
There is a blue point \(e\) such that the red-priority simplification
of \(N/e\) satisfies (C2) and has at least
\((q-2)n/q+1-t\) blue points.
\end{lemma}

\begin{proof}
Choose an affine representation of \(N\); the blue set has affine
dimension \(q-1\), since it spans rank \(q\). The degree bound above
gives a blue point \(e\) incident to at least \((q-2)n/q+1\)
blue-spanned lines.

The blue-spanned lines through \(e\) give distinct parallel classes
in \(N/e\). At most \(t\) of these classes contain a red element,
so at least \((q-2)n/q+1-t\) remain blue after simplification.
Since \(e\) is blue, contraction preserves (C2).
\end{proof}

Let \(\T_L(N)\) denote the \emph{principal truncation} of a rank-two flat \(L\)
in a loopless matroid \(N\). Its rank function is
\begin{equation}\label{eq:truncation}
    \rk_{\T_L(N)}(X)=
    \begin{cases}
        \rk_N(X)-1,&L\subseteq\cl_N(X),\\
        \rk_N(X),&\text{otherwise}.
    \end{cases}
\end{equation}
For a complex representation, choose an added vector
in the two-dimensional span of \(L\), outside the spans of all
ground-set subsets that do not span \(L\). Such a choice exists because
these spans meet the span of \(L\) in finitely many proper subspaces
and the field is infinite. Freely adding this point on \(L\) and
contracting it gives the displayed rank function, so principal
truncation preserves complex representability.
It creates no loops, since a rank-one set cannot span \(L\).
The rank formula also gives
\(\cl_N(X)\subseteq\cl_{\T_L(N)}(X)\) for every \(X\), so every
flat after truncation was a flat before truncation.

\begin{lemma}[{\cite[Lemma~2.4(2)]{CKL}}]\label{lem:repair}
Suppose that a colored loopless matroid \(N\) satisfies (C2).
If every rank-\((\rk(N)-2)\) flat skew to a line \(L\) contains a red
element, then \(\T_L(N)\), with the same ground-set coloring,
also satisfies (C2).
\end{lemma}

\begin{proof}
Put \(r=\rk(N)\), and let \(H\) be a hyperplane of \(\T_L(N)\).
Then \(H\) is an old flat. If it contains \(L\), the rank formula
gives \(\rk_N(H)=r-1\), so it contains a red element by (C2).
Otherwise \(\rk_N(H)=r-2\). If \(\rk_N(H\cup L)=r-1\), then for
any \(x\in L\setminus H\), the set \(H\cup\{x\}\) spans \(L\)
in \(N\) and has rank \(r-2\) after truncation. This contradicts
flatness of \(H\) after truncation. Hence \(\rk_N(H\cup L)=r\),
so \(H\) is skew to \(L\) and contains a red element by hypothesis.
\end{proof}

The following rank identity is a specialization of the accounting
in~\cite[Proposition~4.3]{CKL}. We include a direct proof for sequences
of line truncations.

\begin{lemma}\label{lem:budget}
Let \(N_0\) be a simple loopless matroid, and let \(N_D\) be obtained
from \(N_0\) by \(D\) successive principal truncations of lines,
retaining the ground set. Let \(\Q\) be the collection of nontrivial
parallel classes of \(N_D\). For \(Q\in\Q\), put
\(c_Q=\rk_{N_0}(Q)-1\). Then \(c_Q\ge1\) and
\[
    \sum_{Q\in\Q}c_Q=D.
\]
\end{lemma}

\begin{proof}
As observed above, every flat after a principal truncation is a flat
of the preceding matroid. Thus every final parallel class \(Q\) is a
flat at each earlier stage.

Let \(L_j\) be the line truncated at step \(j\).
By~\cref{eq:truncation}, the rank of \(Q\) drops at this step exactly
when \(L_j\subseteq Q\). Since its final rank is one,
\[
    \rk_{N_0}(Q)-1
    =\bigl|\{j:L_j\subseteq Q\}\bigr|.
\]
Each \(L_j\), once truncated, has rank one and retains rank one
thereafter, because principal truncation of a line cannot lower
the rank of a rank-one set. It therefore lies in exactly one final
parallel class. That class is nontrivial. Conversely, since \(N_0\)
is simple, every nontrivial final class has initial rank at least two.
Summing counts each truncation exactly once.
\end{proof}

\section{Proof of the higher-dimensional bounds}\label{sec:higher-bounds}

The planar bound starts a simultaneous induction on rank for the size of blue
flats and entire admissible configurations.
After reducing to a chosen blue flat, we contract a point supplied by the
degree bound \cref{lem:dirac} and use principal truncations to repair failures of
admissibility.
The principal truncation argument results in parallel classes, and we use induction to bound the number of original blue points in each parallel class.
The rank identity \cref{lem:budget} lets us combine these
bounds with the bound for the remaining blue points. Comparing the
result with the number of blue points retained by the initial
contraction closes the induction.

Put
\[
    c=8-4\sqrt2,\qquad
    \gamma=\frac{3-c}{6}=\frac{4\sqrt2-5}{6}.
\]

\begin{theorem}\label{thm:flats}
Let $(R,B)$ be an admissible coloring of a simple complex-representable
matroid $M$ of rank $r$, and let $t=|R|$. For every rank-$q$ flat
$F$ of $M|B$, where $q\ge2$,
\[
    |F|\le
    \begin{cases}
        t,&q=2,\\[2pt]
        \displaystyle 1+\frac{(q-2)(3q-1)}2(t-1)
            -\gamma q(q-1)t,&q\ge3.
    \end{cases}
\]
Moreover,
\[
    |B|\le\frac{r(3r-7)}2(t-1)-\gamma r(r-1)t.
\]
\end{theorem}

By \cref{lem:three-red}, $t\ge3$. For such $t$ and \(q\ge3\), define
\begin{align}
    U_q(t)&=1+\frac{(q-2)(3q-1)}2(t-1)-\gamma q(q-1)t,
        \label{eq:U}\\
    V_q(t)&=U_q(t)-t
        =\frac{q(3q-7)}2(t-1)-\gamma q(q-1)t.
        \label{eq:V}
\end{align}
Both functions increase with \(t\). We will prove that every
rank-\(q\) blue flat has at most \(U_q(t)\) points and that an
admissible matroid of rank \(q\) has at most \(V_q(t)\) blue points.
These are the two bounds in \cref{thm:flats}.

The auxiliary function \(H_t\) bounds the number of original blue points in each parallel class merged during the truncation stage of the argument.
Since it is superadditive, we can combine the contributions of all merged parallel classes and the remaining blue points.
For fixed \(t\), let
\[
    H_t(0)=0,\qquad H_t(j)=U_{j+2}(t)-1\quad(j\ge1).
\]
Writing
\[
    a=\frac32(t-1)-\gamma t,\qquad
    b=\frac52(t-1)-3\gamma t,
\]
we have \(H_t(j)=aj^2+bj-2\gamma t\) for \(j\ge1\).
Since \(t\ge3\), we have \(0\le\gamma t\le(t-1)/6\).
Thus \(a\ge4(t-1)/3>0\), \(b\ge2(t-1)>0\), and
\(H_t(1)\ge3(t-1)\ge1\), so \(H_t\) is increasing.
For positive integers \(i,j\),
\[
    H_t(i+j)-H_t(i)-H_t(j)=2aij+2\gamma t\ge0,
\]
so \(H_t\) is superadditive on the nonnegative integers.
Direct calculation also gives, for \(q\ge4\),
\begin{equation}\label{eq:recurrence}
    U_q(t)
    =t+\frac{q}{q-2}\bigl(U_{q-1}(t)+t-2\bigr)
    \ge t+U_{q-1}(t).
\end{equation}
The initial values are
\begin{equation}\label{eq:initial}
    V_3(t)=ct-3,\qquad
    U_3(t)=(c+1)t-3\ge2t.
\end{equation}

\begin{proof}[Proof of \cref{thm:flats}]
We prove the blue-flat and ambient-rank bounds simultaneously by
induction on \(q\), uniformly over all admissible ambient matroids
representable over $\mathbb C$. The rank-two blue-flat bound was
established after \cref{lem:complement}; rank-one blue flats have
one point.

First consider any simple rank-three configuration satisfying (C2)
with a spanning blue set of size \(n\) and at most \(t\) red points.
The blue points are noncollinear, and every line spanned by two
of them contains a red point. Thus \cref{thm:planar} gives
\(n\le ct-3=V_3(t)\), proving the ambient-rank bound.

Now let \(F\) be a rank-three blue flat in an admissible matroid.
Contract a red-free complement from \cref{lem:complement}, restrict
to the images of \(F\cup R\), and simplify with red priority.
All ranks inside \(F\) are preserved, and at most \(t\) of its
points become red. If the remaining blue points span rank three,
the preceding bound gives \(|F|\le t+V_3(t)=U_3(t)\).
Otherwise, their original representatives have rank at most two
in \(M|B\), so the line or point bound gives
\(|F|\le2t\le U_3(t)\). This establishes both initial bounds.

Let \(q\ge4\), let \(F\) be a rank-\(q\) flat of \(M|B\), and put
\(m=|F|\). Contract a red-free complement \(P\) from
\cref{lem:complement}, restrict to the images of \(F\cup R\), and
simplify with red priority. Denote the resulting rank-\(q\) matroid
by \(N\), and its blue set by \(B_N\). It satisfies (C2) and has at
most \(t\) red points.

Contraction of \(P\) preserves all ranks inside \(F\), so its points
remain distinct. Each red image can coincide with at most one point
of \(F\). Hence
\[
    |B_N|\ge m-t.
\]
If \(B_N\) does not span \(N\), its original representatives lie in
a flat of \(M|B\) of rank at most \(q-1\). Since \(U_j(t)\) increases
with \(j\ge3\) and \(U_3(t)\ge2t\), induction and the line or point
bound give
\[
    m\le t+U_{q-1}(t)\le U_q(t).
\]
Thus we may assume that \(B_N\) spans \(N\).

Use \cref{lem:dirac} to contract a blue point \(e\), and simplify
with red priority. Let \(N_0\) be the resulting simple matroid,
and let \(B_0\) be its blue set, of size \(n_0\). Then \(N_0\)
has rank \(q-1\), satisfies (C2), and has at most \(t\) red points.
Moreover,
\begin{equation}\label{eq:lower}
    n_0\ge\frac{q-2}{q}(m-t)+1-t.
\end{equation}
Choose an original representative in \(F\) for each point of \(B_0\),
and also denote the original representative of \(e\) by \(e\).

Starting from \(N_0\), perform the following principal truncations
without simplifying, retaining the ground set and its coloring.
Interpret (C3) in the red-priority simplification of the current matroid.
If a line \(\ell\) of its blue restriction witnesses failure of (C3),
choose representatives of its points and let \(L\) be their closure
in the current unsimplified matroid. This is an ambient rank-two flat,
including all parallel labels on the line. Taking ambient closure
does not change skewness, and flats correspond under simplification.
Thus every complementary flat skew to \(L\) contains a red label.
Principally truncate \(L\); by \cref{lem:repair}, this preserves (C2).
Repeat until (C3) holds or at most one blue point remains.
The process stops by rank two, where (C2) forces every point to be red.
Let \(D\) be the number of truncations and let \(s\) be the terminal
rank. Let \(\Q\) contain all nontrivial final parallel classes,
including red classes containing original blue labels.
By \cref{lem:budget},
\begin{align*}
    s&=q-1-D\ge2,\\
    \sum_{Q\in\Q}c_Q&=D,
    \qquad c_Q=\rk_{N_0}(Q)-1\ge1.
\end{align*}

For \(Q\in\Q\), lift the points of \(B_0\cap Q\) to their chosen
representatives in \(F\), and include \(e\). Lifting through the
contraction of \(e\) adds at most one rank, and contraction of \(P\)
preserved all ranks inside \(F\). Thus these original blue points
have rank at most \(c_Q+2\), where
\[
    c_Q+2\le D+2=q+1-s\le q-1.
\]
The point \(e\) is distinct from all the chosen representatives.
If their union with \(e\) has rank at least three, induction
applied to its closure in \(M|B\) therefore gives
\[
    |B_0\cap Q|+1\le U_{c_Q+2}(t).
\]
If that rank is at most two, the line or point bound instead gives
\(|B_0\cap Q|\le t-1\le H_t(c_Q)\).
In all cases,
\begin{equation}\label{eq:class}
    |B_0\cap Q|\le U_{c_Q+2}(t)-1=H_t(c_Q).
\end{equation}

Let \(S_0\) consist of the elements of \(B_0\) whose final parallel
classes are singletons, and let \(B_{\mathrm{term}}\) be the blue
point set of the terminal red-priority simplification.
Every element of \(S_0\) remains blue, so
\[
    n_0=|S_0|+\sum_{Q\in\Q}|B_0\cap Q|,
    \qquad |S_0|\le |B_{\mathrm{term}}|.
\]
If the terminal simplification has at least two blue points, it is
admissible, and its blue set spans its rank \(s\). By the simultaneous
ambient-rank induction and monotonicity in the red count,
\[
    |B_{\mathrm{term}}|
    \le V_s(t)=H_t(s-2)-(t-1)\le H_t(s-2).
\]
If \(s\ge3\) and at most one blue point remains, the same final
inequality follows from \(H_t(s-2)\ge1\). If \(s=2\), there are
no blue points, and \(H_t(0)=0\).

Applying these bounds and superadditivity gives
\begin{align*}
    n_0
    &\le H_t(s-2)+\sum_{Q\in\Q}H_t(c_Q)\\
    &\le H_t\left((s-2)+\sum_{Q\in\Q}c_Q\right)\\
    &=H_t(q-3)=U_{q-1}(t)-1.
\end{align*}
A surviving nontrivial blue class can contribute to both terms
of the first bound, which only increases the upper estimate.
Combining this with \cref{eq:lower} and \cref{eq:recurrence} yields
\[
    m\le
    t+\frac{q}{q-2}\bigl(U_{q-1}(t)+t-2\bigr)=U_q(t).
\]

For the ambient-rank bound, take \(F=B\) when \(\rk(M)=q\).
Then \(P=\varnothing\) and \(N=M\), so \(B_N=B\) spans \(N\)
and no blue point is lost before contracting \(e\).
Consequently, \cref{eq:lower} improves to
\[
    n_0\ge\frac{q-2}{q}|B|+1-t.
\]
The same upper bound on \(n_0\) now gives
\[
    |B|\le
    \frac{q}{q-2}\bigl(U_{q-1}(t)+t-2\bigr)
    =U_q(t)-t=V_q(t).
\]
This completes the simultaneous induction.
\end{proof}

\begin{proof}[Proof of \cref{thm:main}]
Apply the ambient-rank bound in \cref{thm:flats} with $r=k+1$ and
expand. For the coarser bound, put $A=7-2\sqrt2<6$. Then
\[
    (k+1)(Ak-6)=Ak^2+(A-6)k-6\le Ak^2,
\]
and the subtracted constant $(k+1)(3k-4)/2$ is positive for $k\ge2$.
\end{proof}

\section{A construction from cyclic Dowling geometries}
\label{sec:construction}

Let \(\mu_m=\{\zeta\in\mathbb C:\zeta^m=1\}\), and let
\(e_1,\ldots,e_r\) be the standard basis of \(\mathbb C^r\).
The cyclic Dowling geometry \(Q_r(m)\) is represented by the
projective point set
\[
    P=\{[e_i]:1\le i\le r\}
      \cup\{[e_i-\zeta e_j]:1\le i<j\le r,\ \zeta\in\mu_m\}.
\]
Here \([v]\) denotes the projective point represented by \(v\).
These geometries have rank \(r\) and \(r+m\binom r2\) points;
see Dowling~\cite{Dowling1973} and
Kroeker and Lund~\cite[Section~4]{KLFrame}.
Note that Campbell, Kroeker, and Lund~\cite[Proposition 2.1]{CKL} give a slightly better construction for $b_k(3)$; our contribution is the uniform order $k^2t$ for arbitrary $t$.

\begin{proposition}\label{prop:dowling-lower}
For all integers \(k\ge2\) and \(t\ge3\),
\[
    b_k(t)\ge
    (t-2)\left(\binom{k+1}{2}-1\right)+k-1.
\]
\end{proposition}

\begin{proof}
Put \(r=k+1\) and \(m=t-2\), and let \(M\) be the matroid
represented by \(P=Q_r(m)\). Color the points
\[
    R=\{[e_1],[e_2]\}
      \cup\{[e_1-\zeta e_2]:\zeta\in\mu_m\}
\]
red and all remaining points blue. Thus \(R\) consists of the
configuration points on the coordinate line spanned by \(e_1,e_2\),
and
\[
    |R|=m+2=t,\qquad
    |B|=m\left(\binom r2-1\right)+r-2.
\]
We verify admissibility. The blue points represented by
\(e_3,\ldots,e_r,e_1-e_3,e_2-e_3\) span \(\mathbb C^r\),
so (C1) holds.

For (C2), let a hyperplane of \(M\) be represented by the points of
\(P\) in the kernel of a nonzero linear functional
\(a=(a_1,\ldots,a_r)\). Put \(Z=\{i:a_i=0\}\), and partition
the remaining indices into classes under
\(i\sim j\) if \(a_i/a_j\in\mu_m\).
If there are \(c\) such classes, the configuration points in
\(\ker a\) span a space of dimension
\[
    |Z|+\sum_C(|C|-1)=r-c.
\]
Indeed, the coordinate points indexed by \(Z\) span a space of
dimension \(|Z|\). Within each nonzero class \(C\), the vectors
\(e_i-(a_i/a_j)e_j\) span a space of dimension \(|C|-1\),
and there are no configuration points joining different classes
or a zero index to a nonzero index. Since the hyperplane has
rank \(r-1\), we have \(c=1\).
If \(a_1=0\) or \(a_2=0\), the hyperplane contains a red
coordinate point. Otherwise \(a_1/a_2\in\mu_m\), so it contains
the red point \([e_1-(a_1/a_2)e_2]\).

For (C3), let \(W\) be the two-dimensional vector space representing
a line of \(M|B\), and put \(U=\langle e_1,e_2\rangle\).
Let \(f_1,f_2\) be the standard basis of \(\mathbb C^2\).
There is a linear map \(\pi:\mathbb C^r\to\mathbb C^2\) with
\(\pi(e_1)=f_1\), \(\pi(e_2)=f_2\), and \(\pi|_W\) injective:
extend the prescribed map on \(U\cap W\) to an isomorphism on
\(W\), and then extend the resulting map on \(U+W\) to
\(\mathbb C^r\).

Choose a basis \(u,v\) of \(W\). The nonzero determinant
\(\det(\pi(u),\pi(v))\) is affine-linear in each column
\(\pi(e_i)\) separately. An affine-linear function on
\(\mathbb C^2\) that vanishes at \(0,f_1,f_2\) vanishes
identically. Hence, successively for \(i=3,\ldots,r\), we can
replace \(\pi(e_i)\) by one of \(0,f_1,f_2\) while keeping this
determinant nonzero. After these replacements, \(\pi|_W\)
is still injective, and \(\ker\pi\) has the basis
\[
    g_i=
    \begin{cases}
        e_i,&\pi(e_i)=0,\\
        e_i-e_1,&\pi(e_i)=f_1,\\
        e_i-e_2,&\pi(e_i)=f_2,
    \end{cases}
    \qquad 3\le i\le r.
\]
Each \([g_i]\) is blue. Since \(\pi|_U\) is an isomorphism,
\(\ker\pi\) contains no red point. Thus
\(G=\{[x]\in P:x\in\ker\pi\}\) is a red-free flat of \(M\)
of rank \(r-2\). The injectivity of \(\pi|_W\) gives
\(W\cap\ker\pi=\{0\}\), so \(G\) is skew to the given blue
line. This proves (C3), and the formula for \(|B|\) gives the
lower bound.
\end{proof}

When \(t=4\), we have \(m=2\) and \(\mu_m=\{1,-1\}\), so
the same configuration is real-representable. Its coloring is
admissible over \(\mathbb R\), since the represented matroid is
unchanged by extending scalars to \(\mathbb C\). In this case,
\[
    b_k^{\mathbb R}(4)\ge
    2\left(\binom{k+1}{2}-1\right)+k-1=(k-1)(k+3)
    \qquad(k\ge2).
\]

\begin{corollary}\label{cor:complex-order}
For all integers \(k\ge2\) and \(t\ge3\),
\[
    \frac16 k^2t\le b_k(t)\le\frac{7-2\sqrt2}{3}k^2t.
\]
In particular, \(b_k(t)=\Theta(k^2t)\) uniformly in
\(k\ge2\) and \(t\ge3\).
\end{corollary}

\begin{proof}
The upper bound is \cref{thm:main}. For the lower bound, apply
\cref{prop:dowling-lower} and use
\(t-2\ge t/3\) and \(\binom{k+1}{2}-1\ge k^2/2\).
\end{proof}

\bibliographystyle{plain}
\bibliography{quadratic-ckl-bound}
\end{document}